\documentclass[12pt, reqno]{amsart}
\usepackage{amsmath,amsfonts,amssymb,amsthm,amscd}
\usepackage{mathrsfs}
\usepackage[all]{xy}

\usepackage{color}
\usepackage{colordvi}

\newcommand{\on}{\operatorname}
\newcommand{\bb}{\mathbb}
\newcommand{\cal}{\mathcal}
\newcommand{\f}{\mathfrak}
\newcommand{\mbf}{\mathbf}

\def\C{{\mathbb C}}

\def\Z{{\mathbb Z}}

\def\N{{\mathbb N}}
\def\1{{\bf 1}}

\def \<{\langle}
\def \>{\rangle}
\def \w{\omega}

\def \sl{\frak{sl}}

\def \h{\mathfrak{h}}

\def \w{\omega}
\numberwithin{equation}{section}
\newtheorem{theorem}{Theorem}[section]
\newtheorem{prop}[theorem]{Proposition}
\newtheorem{lem}[theorem]{Lemma}
\newtheorem{cor}[theorem]{Corollary}

\theoremstyle{definition}
\newtheorem{remark}[theorem]{Remark}

\newtheorem{example}[theorem]{Example}

\begin{document}
\title[Vertex Operator Algebras]{The varieties of semi-conformal vectors  of affine vertex operator algebras }

\author{ Yanjun Chu}
%\thanks{Jiang is supported by China NSF grants 10931006, 11371245,  China RFDP grant 2010007310052, and the Innovation Program of Shanghai Municipal Education Commission (11ZZ18)}
\address[Chu]{1.School of Mathematics and Statistics,  Henan University, Kaifeng, 475004, China\\
2. Institute of Contemporary Mathematics, Henan University, Kaifeng
475004, China}
\email{chuyj@henu.edu.cn}

\author{Zongzhu Lin}
\address[Lin]{ Department of Mathematics, Kansas State University, Manhattan, KS 66506, USA}
\email{zlin@math.ksu.edu}
\begin{abstract}
 This is a continuation of our work to understand vertex operator algebras using the geometric properties of varieties attached to vertex operator algebras. For a class of  vertex operator algebras including affine vertex operator algebras associated to a finite dimensional simple Lie algebra $\f{g}$,  we describe their varieties of semi-conformal vectors by some matrix equations. These matrix equations are too complicated to be solved for us.  However, for  affine vertex operator algebras associated to the simple Lie algebra $\f{g}$,
we find the adjoint group $G$ of $\f{g}$ acts on the corresponding varieties by a natural way, which implies that  such varieties should  be described more clearly by studying the corresponding  $G$-orbit structures. Based on above methods for general cases, as an example,  considering  affine vertex operator algebras associated to the Lie algebra $\f{sl}_2(\C)$, we shall give the decompositions of $G$-orbits of varieties of their semi-conformal vectors according to different levels. Our results imply that  such  orbit structures depends on the levels of affine vertex operator algebras associated to a finite dimensional simple Lie algebra $\f{g}$.\end{abstract}

\maketitle
\section{Introduction}
\subsection{}
This is a continuation of study of vertex operator algebras using the geometry of the variety of semi-conformal vectors, started in \cite{JL2} and \cite{CL}.

Given a verex operator algebra $(V, Y, \mbf{1}, \omega)$, a vertex operator subalgebra $(U, Y, \mbf{1}, \omega')$ in general has a different conformal vector $\omega'\neq \omega$. The subalgebra $U$ (respectively, $\omega'$) is called {\em semi-conformal}~if $ \omega'_n|_U=\omega_n|_U$ for $n\geq 0$ using the notation $ Y(v, z)=\sum_{n}v_n z^{-n-1}$ for all $ v\in V$. It follows from the definition $ \omega'\in V_2$. The set $\on{Sc}(V, \omega)$ of all semi-conformal vectors is a Zariski closed subset of $ V_2$(\cite[Theorem 1.1]{CL}).

In \cite{CL}, we use the geometry of the set of semi-conformal vectors together with the relations of the corresponding subalgebras  to give two characterizations of Heisenberg vertex operator algebras. In the Heisenberg algebra case,  the finite dimensional abelian Lie algebra $ \f{h}$ is equipped with a  non-degenerated symmetric bilinear form. The automorphism group of $\f{h}$ (preserving the form) is the complex orthogonal group which acts on the set $\on{Sc}(V, \omega)$.  In this paper, we study the geometric structures of $\on{Sc}(V, \omega)$  for affine vertex operator algebras $V=V_{\widehat{\f{g}}}(\ell, 0)$ and $L_{\widehat{\f{g}}}(\ell, 0)$, where $L_{\widehat{\f{g}}}(\ell, 0)$ is the simple quotient of $V_{\widehat{\f{g}}}(\ell, 0)$,  associated to a finite dimensional simple Lie algebra $\f{g}$. In this case, the adjoint group $G$ of $\f{g}$ also acts on the variety $\on{Sc}(V, \omega)$. We shall describe  $G$-orbit structures.
\subsection{}
 Denoted by $\on{ScAlg}(V,\omega)$ the set of semi-conformal subalgebras of a vertex operator algebra $V$ with  conformal vector $\omega$. Obviously, $(U,\omega')\longmapsto \omega'$ defines a surjection from
$\on{ScAlg}(V,\omega)$ to $\on{Sc}(V,\omega)$, denoted it by $\pi$. The commutant of  a subalgebra in $V$ gives an operator $C_V$ on $\on{ScAlg}(V,\omega)$
 as follows
 $$
 \begin{array}{lllll}
 C_V:\on{ScAlg}(V,\omega)\longrightarrow \on{ScAlg}(V,\omega)\\
 \hspace{1.9cm}(U,\omega')\longmapsto (C_{V}(U), \w-\w').
 \end{array}
 $$
 And  $C_V$ induces an involution operator $\omega-$ on $\on{Sc}(V,\omega)$ defined by 
$\omega'\longmapsto \omega-\omega'.$
In fact, these maps form a following commutative diagram
 $$\begin{CD}
 \on{ScAlg}(V,\omega)@>C_V>> \on{ScAlg}(V,\omega)\\
 @VV\pi V      @VV\pi V\\
 \on{Sc}(V,\omega)@>\omega->> \on{Sc}(V,\omega).
 \end{CD}
 $$
 
 Semi-conformal subalgebras and semi-conformal vectors are used to study level-rank duality in vertex operator algebras (\cite{JL2}). Such subalgebras are motivated by \cite[Theorem 3.11.12]{LL} and \cite [Theorem 5.1]{FZ} to determine duality pairs of vertex operator subalgebras in Howe duality \cite{LL1,CHZ}, and depend on coset constructions of vertex operator algebras.

\subsection{}
For the surjection $\pi$, the fiber $\pi^{-1}(\omega')$ for each $\omega'\in \on{Sc}(V,\omega)$ is the set of semi-conformal subalgebras with  conformal vector $\omega'$ of $V$. In  $\pi^{-1}(\omega')$, there exists a unique maximal one $U(\omega')$, which is the maximal conformal extension (in $V$) of semi-conformal subalgebras with 
conformal vector $\omega'$ and can be realized as the double commutant $C_{V}(C_{V}(\<\omega'\>))$ of $\<\omega'\>$ in $V$, where $\<\omega'\>$ is   the Virasoro vertex operator algebra generated by $\omega'$.  In fact,
semi-conformal vectors classify semi-conformal subalgebras of $V$ up to conformal extension. Thus,  for each $\omega'\in \on{Sc}(V,\omega)$, by studying the fiber $\pi^{-1}(\omega')$, we can classify the set $\on{ScAlg}(V,\omega)$.
For $\omega',\omega''\in \on{Sc}(V,\omega)$, we define a partial order $\omega'\preceq\omega''$ if $U(\omega')\subseteq U(\omega'')$(\cite[Definition 2.7]{CL}). In Heisenberg algebra case, we described the structure of  $\on{Sc}(V,\omega)$ as $G=\on{Aut}(V, \omega)$-orbits . Moreover, for a vertex operator algebra $(V,\omega)$, when $\on{Sc}(V, \omega)$ satisfies some properties, we characterized   $(V,\omega)$  as a Heisenberg  vertex operator algebra with conformal vector $\omega_{0}$(cf.  \cite[Theorem 1.4 and Theorem 1.5]{CL}). We also note that the group $G=\on{Aut}(V, \omega)$ acts on both $\on{ScAlg}(V, \omega)$ and $\on{Sc}(V, \omega)$ and the map $\pi$ is $G$-equivariant.
\subsection{}
For  affine vertex operator algebras $V=V_{\widehat{\f{g}}}(\ell, 0)$ and $L_{\widehat{\f{g}}}(\ell, 0)$ associated to  a  finite dimensional  simple Lie algebra $\f{g}$,
 we  choose an orthonormal basis of the Lie algebra $\f{g}=V_1$.  Then a vector belongs to $\on{Ker}L(1)\cap V_2$ if and only if there exists a symmetric matrix such that  it can be expressed  as a quadratic form  with respect to the given orthonormal basis  of $\f{g}$ (See Proposition 3.1 in  Section 3).  This proposition  is critical  for us to find semi-conformal vectors of $(V,\omega)$.  With the help of this proposition,  we characterize  semi-conformal vectors of $(V,\omega)$  by some matrix equations (See equations (\ref{e3.4})--(\ref{e3.8}) in Section 3).  But these matrix equations obtained are too complicate to be solved in general. As Heisenberg case indicated in \cite{CL}, the variety $\on{Sc}(V, \omega)$ is closely  related to the affine Lie algebra with a fixed level as well as the lattice structure. For affine vertex operator algebras $V=V_{\widehat{\f{g}}}(\ell, 0)$ and $L_{\widehat{\f{g}}}(\ell, 0)$,  considering the action of the adjoint group $G$ of $\f{g}$ on the variety $\on{Sc}(V, \omega)$, we expect to give decompositions of  $G$-orbits of  $\on{Sc}(V, \omega)$.
 
 %\subsection{} 
 To get $G$-orbits of  $\on{Sc}(V, \omega)$ for affine vertex operator algebras 
 $V=V_{\widehat{\f{g}}}(\ell, 0)$ and $L_{\widehat{\f{g}}}(\ell, 0)$ associated to  a  finite dimensional  simple Lie algebra $\f{g}$, in this paper, we shall  consider  affine vertex operator algebras $V_{\widehat{\f{sl}}_2}(\ell,0)$ and $L_{\widehat{\f{sl}}_2}(\ell,0)$ as examples. At first, we study the variety $\on{Sc}(V,\omega)$  for  $V=V_{\widehat{\f{sl}}_2}(\ell,0)$. When the matrix equations  (\ref{e3.4})--(\ref{e3.8}) in Section 3 are applied  to $V_{\widehat{\f{sl}}_2}(\ell,0)$, we can obtain matrix equations describing $\on{Sc}(V_{\widehat{\f{sl}}_2}(\ell,0),\omega)$. On the other hand, since the adjoint group $\on{PSL_2(\C)}$ of $\f{sl}_2(\C)$ is isomorphic to the special orthogonal group $\on{SO}_3(\C)$, then  the $\on{PSL_2(\C)}$-action on  $\on{Sc}(V_{\widehat{\f{sl}}_2}(\ell,0),\omega)$ can be realized by the group $\on{SO}_3(\C)$ in a natural way.  Moreover, we find that the orbits of $\on{Sc}(V_{\widehat{\f{sl}}_2}(\ell,0),\omega)$ are determined completely by solutions with diagonal matrices  in above mentioned matrix equations. Thus, the problems are reduced to find  all solutions with the form of diagonal matrices from obtained matrix equations. Finally, we obtain the decomposition of $\on{PSL_2(\C)}$-orbits of the varieties $\on{Sc}(V_{\widehat{\f{sl}}_2}(\ell,0),\omega)$ for the level $\ell \neq -2$.  Such decompositions can be summarized  as follows

1) If $\ell=0,1$,  then $\on{Sc}(V_{\widehat{\f{sl}}_2}(\ell,0),\omega)=\{0\}\cup\{\omega\};$

2) If $\ell\neq -2,0,1,4$, then 
$\on{Sc}(V_{\widehat{\f{sl}}_2}(\ell,0),\omega)=\on{Orb}_{\omega_{A_1}}\cup \on{Orb}_{\omega_{A_4}}\cup\{0\}\cup\{\omega\},$
 where $\on{Orb}_{\omega_A}=\{\omega_{oAo^{tr}}|o\in \on{SO}_3(\C)\}$ and 
$$
A_1=\left(
\begin{array}{llll}
&\frac{1}{2\ell}&0&0\\
&0&0&0\\
&0&0&0
\end{array}
\right),
A_4=\left(
\begin{array}{llll}
&-\frac{1}{\ell(\ell+2)}&0&0\\
&0&\frac{1}{2(\ell+2)}&0\\
&0&0&\frac{1}{2(\ell+2)}
\end{array}
\right);$$

3) If $\ell=4$, then 
$\on{Sc}(V_{\widehat{\f{sl}}_2}(\ell,0),\omega)=
\{0\}\cup\{\omega\}\cup
\bigcup_{\omega_A\in T}\on{Orb}_{\omega_A},$ where 
$$T=\left\{
A=\left(
\begin{array}{llll}
&a_1&0&0\\
&0&a_2&0\\
&0&0&a_3
\end{array}
\right)\Bigg |tr(A)=\frac{1}{8}~and~
 a_1a_2+a_1a_3+a_2a_3=0\right\}\Bigg /S_3.$$
These results shows that  for $\ell\neq -2,0,1,4$,  
$\on{Sc}(V_{\widehat{\f{sl}_2}}(\ell,0),\omega)$ is 
the union of  finitely  many $\on{PSL_2(\C)}$-orbits and 
has two nontrivial $\on{PSL_2(\C)}$-orbits. 
However, $\on{Sc}(V_{\widehat{\f{sl}}_2}(4,0),\omega)$ 
consists of infinitely many $\on{PSL_2(\C)}$-orbits. 
Moreover, for the simple quotients  
$(L_{\widehat{\f{sl}}_2}(\ell,0),\omega)$ of 
 $V_{\widehat{\f{sl}}_2}(\ell,0) $  for the level 
$\ell \in \Z_{+}$,  we  also prove that they have  the same
 decompositions  as 
 $\on{Sc}(V_{\widehat{\f{sl}}_2}(\ell,0),\omega)$.
\subsection{}
 Our results  imply also  some new problems to 
understand  general affine vertex operator algebras and
lattice vertex operator algebras. On one hand, since the 
vertex operator algebra  $L_{\widehat{\f{sl}}_2}(1,0)$ 
is isomorphic to the lattice vertex operator algebra 
$V_{A_1}$ associated to the root lattice of $A_1$ type.  
Then we know  that $\on{Sc} (V_{A_1},\omega)$ consists of 
two trivial $\on{PSL_2(\C)}$-orbits $\{0\}$ and $\{\omega\}$. When we  consider the orbit structures of $\on{Sc}(L_{\widehat{\f{g}}}(1,0),\omega)$ for the cases of  the simple affine vertex operator algebras $L_{\widehat{\f{g}}}(1,0)$  associated to the finite dimensional simple Lie algebras $\f{g}$  of $A,D,E$ types. We can  study $G$-orbit informations of $\on{Sc}(V,\omega)$ for  the lattice vertex operator algebra $V_{Q_{\f{g}}}$ associated to the root lattice of $\f{g}$ because of  $L_{\widehat{\f{g}}}(1,0)\cong V_{Q_{\f{g}}}$.  On the other hand,  it follows from above results that the $\on{PSL_2(\C)}$-orbit structure of $\on{Sc}(V_{\widehat{\f{sl}}_2}(\ell,0),\omega)$  depends on the level $\ell$, which leads us to think of  some new problems such as,
 why has $\on{Sc}(V_{\widehat{\f{sl}}_2}(\ell,0),\omega)$ finitely or infinitely many $\on{PSL_2(\C)}$-orbits for different levels $\ell$?  what is the level $\ell$ taken such that  $\on{Sc}(V_{\widehat{\f{g}}}(\ell,0),\omega) $ has finitely or infinitely many $G$-orbits?
Such problems can help us to understand deeply the affine vertex operator algebra $V$ by geometric properties of  its varieties $\on{Sc}(V,\omega)$.
\subsection{} By \cite[Lemma. 5.1]{JL2}, a vertex subalgebra $U$ of $V$ cannot contain two different conformal vectors $\omega'$ and $\omega''$ such that both $(U, \omega')$ and  $(U, \omega'')$ are semi-conformal vertex operator subalgebras of $ (V, \omega)$. We remark that on a vertex algebra, there can be many different conformal vectors to make it become non-isomorphic vertex operator algebras, even thought they have the same conformal gradations. The result of  \cite[Lem. 5.1]{JL2} says that no one is a semi-conformal with respect to another. \cite[Example 2.5.9]{BF} provides a large number of such examples on the Heisenberg vertex operator algebras.  Thus the map from $\on{ScAlg}(V, \omega)$ to the set of all vertex subalgebras of $V$, forgetting the conformal structure,  is injective.  The conformal vector in $ V$ uniquely determines the semi-conformal structure on a vertex subalgebra of $V$ if there is any.

One of the main  motivations of this work is to investigate  the conformal structure on a vertex subalgebra of a vertex operator algebra.  In conformal field theory, the conformal vector completely determines the conformal structure (the module structure for the Virasoro Lie algebra).  In mathematical physics, a vertex operator algebra has been investigated extensively as a Virasoro module (see \cite{DMZ,D,KL,DLM,LY,M,LS,S,L,Sh}) by virtue of conformal vector.
\subsection{}  This paper is organized as follows: In Sect.2,
we shall review the basic notions and results of semi-conformal subalgebras(vectors) for a vertex operator algebra  according to \cite{JL2,LL,CL};
In Sect.3, we recall some related results of the affine vertex operator algebras  associated to the finite dimensional simple Lie algebra $\f{g}$,  and then give the description of the variety of semi-conformal vectors of $V_{\widehat{\f{g}}}(\ell,0)$ by some matrix  equations;
Moreover, for affine vertex operator algebras $V_{\widehat{\f{g}}}(\ell,0)$ and $L_{\widehat{\f{g}}}(\ell,0)$,
we shall study the action of the adjoint group $G$ of $\f{g}$ on varieties $\on{Sc}(V_{\widehat{\f{g}}}(\ell,0),\omega)$ and $\on{Sc}(L_{\widehat{\f{g}}}(\ell,0),\omega)$.
In Sect.4, as examples, we shall study  $\on{PSL}_2(\C)$-orbit structures of the varieties  $\on{Sc}(V_{\widehat{\f{sl}}_2}(\ell,0),\omega)$ and $ \on{Sc}(L_{\widehat{\f{sl}}_2}(\ell,0),\omega)$  for $\ell \neq -2,4$; In Sect. 5, we shall describe  $\on{PSL}_2(\C)$-orbits of  the varieties $\on{Sc}(L_{\widehat{\f{sl}}_2}(\ell,0), \omega)$ and $\on{Sc}(L_{\widehat{\f{sl}}_2}(\ell,0), \omega)$ for $\ell=4$.
\subsection{Notations:}  $\C$ is  the field of complex number; $\Z$ is the set of integers; $\Z_{+}$ (resp. $\Z_{-})$ is the set of positive (resp. negative) integers; $\N$ is the set of non-negative integers.

{\em Acknowledgement:} This work started when the first author was visiting Kansas State University from September 2013 to September 2014. He thanks the support by Kansas State University and its hospitality.  The first author also thanks China Scholarship Council for their financial supports. The second author thanks C. Jiang for many insightful discussions. This work was motivated from the joint work with her. The second author also thanks Henan University for the hospitality during his visit in the summer of 2016, during which this work was carried out.

\section{ Semi-conformal vectors and semi-conformal subalgebras of a vertex operator algebra}
\setcounter{equation}{0}
\subsection{}
For basic notions and results associated with vertex operator algebras, one is
referred  to the books \cite{FLM,LL,FHL,BF}.  We will use $(V, Y, 1)$ to denote a vertex algebra and $(V, Y, 1, \omega)$ for a vertex operator algebra. When we deal with several different vertex algebras, we will use $ Y^V$, $1^{V}$, and $\omega^V$ to indicate the dependence of the vertex algebra or vertex operator algebra $V$.  For example $Y^V(\omega^V, z)=\sum_{n\in \Z} L^V(n)z^{-n-2}$.  To emphasize the presence of the conformal vector $\omega^V$, we will simply write $(V, \omega^V)$ for a vertex operator algebra and $V$ simply for a vertex algebra (with $Y^V$ and $1^V$ understood). We refer \cite{BF} for the concept of vertex algebras. Vertex algebras need not be graded, while a vertex operator algebra $(V,\omega^V)$ is always $ \Z$-graded by the $L^V(0)$-eigenspaces $V_n$ with integer eigenvalues $ n \in \Z$. We assume that each $ V_n$ is finite dimensional over $\C$ and $ V_n=0$ if $ n<<0$.   
\subsection{}
 In this section, we shall first review semi-conformal vectors (subalgebras) of a vertex operator algebra (\cite{JL2,CL}).

%Let  $(V, Y^V , 1^V )$ and $(W, Y^W, 1^W)$ be two vertex algebras. It follows  from  Section 3.9 of \cite{LL} that
%a homomorphism $f : V\rightarrow W$ of vertex algebras satisfies
%$$f(Y^V (u, z)v) = Y^W(f(u),z)f(v),~~ \forall u, v \in V;~f(1^V ) = 1^W.$$
Let $V$ and $W$ be two vertex operator algebras with conformal vectors
$\omega^V$ and $\omega^W$, respectively. Let $f:V\rightarrow W$ be a homomorphism of vertex algebras. Then  $f$  is called {\em conformal} if $f \circ L^V (n) = L^W(n) \circ f, ~\mbox{for~all}~ n \in\Z,$ i.e., $f(\omega^V ) = \omega^W$.
 We say $f$  is {\em semi-conformal}
if $f \circ L^V (n) = L^W(n) \circ f,$ for all $n\geq 0$.
%We remark that, for any vertex algebra homomorphism $f$ between two vertex operator algebras $(V, Y^V, 1^V,\omega^V)$ and $ (W, Y^W, 1^V, \omega^W)$, one always has $f \circ L^V (-1) = L^W(-1) \circ f$. Also $f$ is conformal if and only if $f \circ L^V (n) = L^W(n) \circ f,$ for all $n\geq -2$.  Thus a  semi-conformal  homomorphism $f$ is conformal if and only if $f \circ L^V (-2) = L^W(-2) \circ f $. 
Let $(V,\omega^V)$ be a vertex subalgebra of $(W,\omega^W)$ and
the map $f: V \rightarrow W$ is the inclusion, we say $V$ is a
conformal subalgebra of $W$ if $f$ is
conformal ($V$ has the same conformal vector with $W$).
If $f$ is semi-conformal, then $(V,\omega^V)$ is called a semi-conformal
subalgebra of $(W,\omega^W)$ and $\omega^V$ is called
a semi-conformal vector of $(W,\omega^W)$.

We remark that if $V$ is of CFT type, then $\omega'\in V_2$ is semi-conformal if and only 
if $\omega'\in Ker(L^V(1))$ and $L'(0)\omega'=2\omega$.
\subsection{}
For a vertex operator algebra $(W,\omega^W)$, we define
$$\begin{array}{lllll}
&\on{ScAlg}(W,\omega^W)=\{(V,\omega^V)\; |\; (V,\omega^V)~ \text{is~a~semi-conformal~subalgebra~of~}(W,\omega^W)\};\\
&\on{Sc}(W,\omega^W)=\{\omega'\in W\; |\;\omega'~\text{is~a~semi-conformal~vector~of}~(W,\omega^W)\};\\
&\overline{\on{S}(W,\omega^W)}=\{(V,\omega')\in \on{ScAlg}(W,\omega^W)|C_{W}(C_{W}(V))=V\},
\end{array}
$$
where $ C_{W}(V)$ is the commutant defined in \cite[3.11]{LL}.  A semi-conformal subalgebra $(U, \omega^U)$ of $W$ is called {\em conformally closed} if $ C_W(C_W(U))=U$ (see \cite{JL2}). So  the set $\overline{S(W,\omega^W)}$ consists of all conformally closed
semi-conformal subalgebras of $(W,\omega^W)$.

It follows from the definition that there is a surjective map $ \on{ScAlg}(W, \omega^W)\rightarrow \on{Sc}(W,\omega^W)$ by $(V, \omega^V)\mapsto \omega^V$.  There is also a surjective map $ \on{ScAlg}(W, \omega^W)\rightarrow \overline{\on{S}(W,\omega^W)}$ defined by $(V, \omega^V)\mapsto (C_W(C_W(V), \omega^V)$. Thus, 
 the restriction of the map $\on{ScAlg}(W,\omega^W)\rightarrow \on{Sc}(W, \omega^W)$ to the set $\overline{\on{S}(W,\omega^W)}$  is a bijection (\cite[Proposition 2.1]{CL}).
 
%\begin{proof} The map $\omega'\mapsto C_{W}(C_W(<\omega'>))$ is the inverse map  %$\on{Sc}(W, \omega^W)\rightarrow \overline{\on{S}(W,\omega^W)}$.
%\end{proof}
%Here $C_{W}(V)$ is the commutant of vertex subalgebra $V$ in $W$, defined as
%$$\{u\in W|
%[Y^{W}(u,z_1), Y^{W}(v,z_2)]=0,~\mbox{for ~all~}v\in V\},$$
%or,
%$$\{u\in W|v_nu=0,~\mbox{for~all~}v\in V~\mbox{and}~n\geq 0\}.$$ $C_{W}(V)$ is also
%a vertex subalgebra of $W$ and has commutant $C_{W}(C_{W}(V))$ in $W$(see \cite{LL} for details).

%In a special case that  $(W,\omega^W)$ is a $\N$-graded vertex operator algebra with $W=\coprod\limits_{n\in \N} W_n$ and $W_0=\C\mathbf{1}$, the condition for a  vertex operator  subalgebra to be semi-conformal is much simpler. In this case,  for any vertex subalgebra $V$ of $W$, if $(V,\omega^{V})$ is a vertex operator algebra for  some $\omega^V\in V$, then
%$(V,\omega^{V})$ is a semi-conformal subalgebra of $W$ if and only if
%$\omega^{V}\in W_2\bigcap \mbox{Ker} L^W(1)$(see \cite[Theorem 3.11.12]{LL}).
%Moreover, if $(V,\omega^{V})$ is a semi-conformal subalgebra of $(W,\omega^{W})$,
%then $(C_{W}(V),\omega^{C}=\omega^W-\omega^V)$ is also a semi-conformal
%subalgebra of $(W,\omega^{W})$, that is, if $\omega^{V}\in \on{Sc}(W,\omega^W)$, then $\omega^C=\omega^{W}-\omega^{V}\in \on{Sc}(W,\omega^W)$. Let $(V,\omega^V), (U,\omega^U)$ be two semi-conformal subalgebras of $(W,\omega^W)$.
%If ~$\omega^V=\omega^U$ and $V\subset U$, then we say $(U,\omega^{U})$ is a conformal extension of $(V,\omega^{V})$ in $(W,\omega^{W})$. Moreover,
Let $(V,\omega^V)$ be a semi-conformal subalgebra of $(W,\omega^W)$.
 Then $(V,\omega^{V})$ has a unique maximal conformal extension
$(C_{W}(C_{W}(V)),\omega^{V})$ in $(W,\omega^{W})$ in the sense that
if $(V,\omega^V)\subset (U, \omega^V)$, then $(U, \omega^V)\subset  (C_{W}(C_{W}(V)),\omega^{V})$( see \cite[Corollary 3.11.14]{LL}).  

%\begin{proof} For each $\omega'\in Sc(W,\omega)$, we first look for a semi-conformal subalgebra
%$(U,\omega')\in \overline{S(W,\omega)}$.

%Let $(<\omega'>,\omega')$ be a simple VOA generated by $\omega'$ in $(W,\omega)$. Then $(<\omega'>,\omega')$
%is a semi-conformal subalgebra of $(W,\omega)$ and  $C_{W}(C_{W}(<\omega'>))$
%is the maximal conformal extension of $(<\omega'>,\omega')$ by Corollary 3.11.14 in \cite{15}.
%Denoted by $(U,\omega')=C_{W}(C_{W}(<\omega'>)).$ Because of the uniqueness of  the
%maximal conformal extension of $(<\omega'>,\omega')$, we have $C_{W}(C_{W}(U))=U$, so $(U,\omega')\in \overline{S(W,\omega)}$.

%Conversely, if a semi-conformal subalgebra $(U,\omega')\in \overline{S(W,\omega)}$, we have $\omega'\in Sc(W,\omega)$ by the definition of $\overline{S(W,\omega)}$.
%\end{proof}
\subsection{}
Let  $(W,\omega^W)$ be a general $\Z$-graded vertex operator algebra. The set $\on{Sc}(W,\omega^W)$ forms an affine algebraic variety(\cite[Theorem 1.1]{CL}. In fact,  a semi conformal  vector $\omega'\in W$ can be  characterized by algebraic equations of degree at most $2$ as described  in \cite[Proposition 2.2]{CL}. The algebraic variety $\on{Sc}(W,\omega^W)$ has also a partial order $\preceq$ (See \cite[Definition 2.7]{CL}), and this partial order  can be characterized by algebraic equations in  \cite[Proposition 2.8]{CL}.%\begin{align}\label{2.1}\left\{\begin{array}{llll}
\subsection{}
 By \cite[Remark 2.10]{CL}, to determine the space 
 $\on{ScAlg}(V,\omega)$, one needs to

(i)  Determine $\on{Sc}(V, \omega)$ and

(ii) For each $\omega'\in \on{Sc}(V, \omega)$, determine  the
 set of conformal subalgebras of $C_{V}(C_{V}(\<\omega'\>))$.

The following lemma is straight forward to verify using the
 definition.

\begin{lem} \label{lem:2.1}Let   
$ \phi: (V, Y^V, \mbf{1}, \omega^V)\rightarrow  
(W, Y^W, \mbf{1}, \omega^W)$ be a conformal vertex operator
 algebra  homomorphism.  For any vertex operator subalgebra  
 $(U, Y, \mbf{1}, \omega')$ of $ V$, 
 $(\phi(U), Y, \mbf{1}, \phi(\omega'))$ is also a vertex 
 operator subalgebra of $(W, Y^W, \mbf{1}, \omega^W)$. Furthermore, $(\phi(U), Y, \mbf{1}, \phi(\omega'))$ is a 
 semi-conformal  (resp.conformal)  subalgebra of  
 $(W, Y^W, \mbf{1}, \omega^W) $ if  $(U, Y, \mbf{1}, \omega')$
is a semi-conformal  (resp.conformal) subalgebra of
$(V, Y^V, \mbf{1}, \omega^V) $.
\end{lem}

\begin{remark} The lemma has following two immediate consequences.  The first consequence is that the conformal homomorphism $ \phi$ restricts to a map
$\on{Sc}(V, \omega^V)\rightarrow \on{Sc}(W, \omega^W)$. This means that the $\on{Sc}(V, \omega^V)$ is functorial in the category of vertex operator algebra with conformal homomorphisms as morphisms. The second consequence is that $\on{Sc}(V,\omega^V)$ and $\on{Scalg}(V,\omega^V)$ are invariant under automorphism group. Let $ G=\on{Aut}(V, \omega^V)$ be the group of all conformal automorphisms of $(V, \omega^V)$. The lemma implies that $ G$ acts on the above three sets and the projection map $ \pi:  \on{ScAlg}(V, \omega^V)\rightarrow \on{Sc}(V, \omega^V)$ is $ G$-equivariant. Thus the question of determining $ \on{Sc}(V, \omega^V)$ can be reduced to determining the $ G$-orbits in $\on{Sc}(V, \omega^V)$. Also, the subgroup $ G_{\omega'}=\on{Stab}_{G}(\omega')$ acts on the fiber $ \pi^{-1}(\omega')$.
\end{remark}

Note that for  a vertex operator algebra $V=\oplus_{n\in \bb{Z}} V_{n}$ with
$V_{n}=\{v \in V\; |\; L^V(0)(v)=nv\}$
being the eigenspace, the group $ G$ also acts on each $ V_n$ making $V_n$ into a finite dimensional $G$-module. If $ \omega' \in \on{Sc}(V, \omega^V)$, then $ L^V(0)\omega'=L'(0)\omega'=2\omega'$. This implies that $ \on{Sc}(V, \omega^V)\subseteq V_2$.
Furthermore, $ L^V(1)\omega'=L'(1)\omega'=0$ implies that $ \on{Sc}(V, \omega^V)\subseteq \ker(L^V(1))\cap V_2$. In particular, $\ker(L^V(1))\cap V_2$ is a $G$-module. One would be interested in determining the $ G$-module structure.
Thus we have
\begin{prop} Let $(V, \omega^V)$ be a vertex operator algebra, then $ \on{Sc}(V, \omega^V)$ is a $ G=\on{Aut}(V)$-stable subset of $ \ker(L^V(1))\cap V_2$.
\end{prop}
\begin{lem} \label {l2.8} Let $(V,\omega^V)$ and $(W,\omega^W)$ be two CFT-type vertex operator algebras, respectively.  If $\rho$ is a conformal vertex operator algebra homomorphism from $V$ to $W$, then
$\rho$ induces a map from $\on{Sc}(V,\omega^V)$ to $\on{Sc}(W,\omega^W)$, denoted by $\widehat{\rho}$. Thus, we have\\
 1) If  $\on{Ker}\rho \cap \on{Sc}(V,\omega^V)=\{0\}$, then $\widehat{\rho}$ is an injective map ;\\
2) If  $\rho$ is surjective and $\on{Ker}\rho \cap (V_1\oplus V_2)=\{0\}$, then $\widehat{\rho}$ is an isomorphism from $\on{Sc}(V,\omega^V)$ to $\on{Sc}(W,\omega^W)$.
\end{lem}
\begin{proof}
The conformal vertex operator algebra homomorphism $\rho$ is a gradation-preserving map and it satisfies
\begin{equation}
\rho(1_V)=1_W;\rho(\omega^V)=\omega^W;
\end{equation}
\begin{equation}
\rho(Y_V(v,z)u)=Y_W(\rho(v),z)\rho(u), ~for ~\forall v,u\in V.
\end{equation}
For a nonzero vector $\omega'\in \on{Sc}(V,\omega^V)$, we know that $\rho(\omega')\in
 \on{Sc}(W,\omega^W)$. If $\on{Ker}\rho \cap \on{Sc}(V,\omega^V)=\{0\}$,  then $\rho(\omega')\neq 0$. Hence $\widehat{\rho}$ is a injective map.

 For  $\omega''\in \on{Sc}(W,\omega^W) $, there is a semi-conformal  subalgebra $(U,\omega'')$ of $W$ such that $\omega''_n=\omega^W_n$ on $U$ for $n\geq 0$. Since $\rho$ is a surjection, there exists a preimage $\omega'$ in $V$ such that 
 \begin{equation}
 \rho(\omega'_n\omega')=\omega''_n\omega''=\omega^W_n\omega''
 =\rho(\omega^V_n\omega'),~for~n\geq 0.
 \end{equation}
 It follows that $\omega'_1\omega'=\omega^V_1\omega'=2\omega'$. Because of  $\on{Ker}\rho \cap (V_1\oplus V_2)=\{0\}$ and  the relations (2.3),
then we have  $\omega'_n\omega'=\omega^V_n\omega'$  for $n\geq 2$.  By Remark 2.2, we also have
$\omega'_0\omega'=\omega^V_0\omega'$.  Thus, by \cite[Proposition 2.2]{CL}, we know that the preimage $\omega'$
of $\omega''\in \on{Sc}(W,\omega^W)$ is a semi-conformal vector of $(V,\omega^V)$. Finally, we associate the conclusion  1) to get that $\widehat{\rho}$ is an isomorphism from $\on{Sc}(V,\omega^V)$ to $\on{Sc}(W,\omega^W)$.
 \end{proof}

%Let $\mbox{Min}\on{Sc}(W,\omega^W)$ be the subset of $\on{Sc}(W,\omega^W)$ which consists of all non-zero minimal semi-conformal vectors of $(W,\omega^W)$ under the partial order $\preceq$.
\subsection{}
In fact, the commutant of
$(W,\omega^W)$  can induce an involution $\omega^W-$ of $\on{Sc}(W,\omega^W)$ as follows
 $$
 \begin{array}{lllll}
 \omega^W-:\on{Sc}(W,\omega^W)\longrightarrow \on{Sc}(W,\omega^W)\\
 \hspace{2.8cm}\omega'\longmapsto \omega^W-\omega'.
 \end{array}
 $$
so for  $\omega^1,\omega^2\in \on{Sc}(W,\omega^W)$, we know 
$\omega^W-\omega^1$ and $\omega^W-\omega^2$ are  conformal vectors of commutants $C_{W}(\langle\omega^1\rangle)$ and  $C_{W}(\langle\omega^2\rangle)$, respectively. 
If $\omega^1\preceq \omega^2$, then $\omega^W-\omega^2\preceq \omega^W-\omega^1$.

\section{Semi-conformal vectors of affine vertex operator algebras}
\setcounter{equation}{0}
For a finite dimensional complex simple Lie algebra $ \f{g}$, the corresponding (untwisted) affine Lie algebra is
$\widehat{\f{g}}=\f{g}\otimes \bb{C}[t, t^{-1}]\oplus \bb{C}K$. Let $ \widehat{\f{g}}_{(\geq 0)}=\f{g}\otimes \bb{C}[t]+\bb{C}K$ be the Lie subalgebra and
$$V_{\widehat{\f{g}}}(\ell, 0)=U(\f{g})\otimes_{U( \widehat{\f{g}}_{(\geq 0)})}\bb{C}_{\ell}$$
be the induced $\widehat{\f{g}}$-module and $ L_{\widehat{\f{g}}}(\ell, 0)$ is the unique irreducible quotient of $ V_{\widehat{\f{g}}}(\ell, 0)$ for all $\ell \in \bb{C}$.
When $\ell$ is not the critical (i.e., $\ell$ is not negative of the dual Coexter number $h^\vee$ of $ \f{g}$), then $ V_{\widehat{\f{g}}}(\ell, 0)$ is a vertex operator algebra and $L_{\widehat{\f{g}}}(\ell, 0)$ is a simple vertex operator algebra
(cf. \cite[6.2]{LL}). We briefly recall  the constructions for later use.

The PBW theorem implies that $ V_{\widehat{\f{g}}}(\ell, 0)$ has a $\bb{C}$-linear basis
\[\{ a^{i_1}(-n_1)\cdots a^{i_r}(-n_r)\mbf{1}\; |\; 1\leq i_1\leq i_2\leq \cdots\leq i_r\leq d, \; n_1\geq 1, \cdots, n_r\geq 1\}\]
if $ \{a^1, \cdots, a^d\}$ is a basis of $ \f{g}$.
In particular,
\[ V_n=\bb{C}\on{-Span}\{ a^{i_1}(-n_1)\cdots a^{i_r}(-n_r)\mbf{1} \; |\; \sum_{i=1}^{d} n_i=n\}. \]
Hence $ V_1=\f{g}(-1)\mbf{1}=\f{g}$ and $V_2$ has a basis
\[\{ a^i(-1)a^j(-1)\mbf{1}\; |\; 1\leq i\leq j\leq d\}\cup \{a^i(-2)\mbf{1}\;|\; i=1,\cdots,d\}.\]
Hence $\dim V_2={d(d+3)/2}$.

Note that conformal vector $ \omega$ is of $ V=V_{\widehat{\f{g}}}(\ell, 0)$ is
\[ \omega=\frac{1}{2(\ell+h^\vee)}\sum_{i=1}^d u^{i}(-1)u^{i}(-1)\mbf{1} \in V_2\]
for any fixed  orthogonal normal basis $\{ u^1, \cdots, u^{d}\}$ of $ \f{g}$ with respect to the normalized Killing form $\langle \cdot, \cdot \rangle $ on $\f{g}$.

 Let $G$ be the adjoint algebraic group corresponding to $ \f{g}$. $G$ acts on $\f{g}$ as automorphism group of the Lie algebra $\f{g}$, and thus preserving the Killing form on $ \f{g}$.  Since $G$ is a simple algebraic group, this action defines an embedding of $G\subset \on{SO}_{\dim \f{g}}(\bb{C})$.  In case that $\f{g}=\f{sl}_2$,  we have $ G=\on{SO}_{3}(\bb{C})$ (both of type $\on{A}_1=\on{B}_1$). This is not true for higher rank Lie algebras. Thus $G$ also acts on $ \widehat{\f{g}}$ as Lie algebra automorphism preserving the Lie subalgebra $\widehat{\f{g}}_{\geq 0}$ and the central element $ K$. Hence $G$ acts on $\widehat{\f{g}}$-module   $V_{\widehat{\f{g}}}(\ell, 0)$ compatible with the action on $ \widehat{\f{g}}$, i.e.,
 the Lie algebra module structure
 \[ \widehat{\f{g}}\otimes V_{\widehat{\f{g}}}(\ell, 0)\rightarrow V_{\widehat{\f{g}}}(\ell, 0)\]
 is a $G$-module homomorphism. The $G$-action on $ V_{\widehat{\f{g}}}(\ell, 0)$  can be directly written by
 \[g(a^{i_1}(-n_1)\cdots a^{i_r}(-n_r)\mbf{1})=g(a^{i_1})(-n_1)\cdots g(a^{i_r})(-n_r)\mbf{1}.\]
Thus $G$ also acts on $V_{\widehat{\f{g}}}(\ell, 0)$ as automorphisms of the vertex operator algebra preserving the conformal vector $\omega$.  Since  $L_{\widehat{\f{g}}}(\ell, 0)$ is the unique $\widehat{\f{g}}$-irreducible quotient of $V_{\widehat{\f{g}}}(\ell, 0)$, $G$ also acts on $L_{\widehat{\f{g}}}(\ell, 0)$.   Let $\pi: V_{\widehat{\f{g}}}(\ell, 0)\rightarrow L_{\widehat{\f{g}}}(\ell, 0)$ be the quotient map. Thus $ \pi $ is a conformal homomorphism of vertex operator algebras, i.e., $\pi(\omega)$ is the conformal vector of $ L_{\widehat{\f{g}}}(\ell, 0)$. Then $ \pi$ is a $G$-module homomorphism.

In particular, the group $ G$ acts on 
$\on{Sc}(V_{\widehat{\f{g}}}(\ell, 0))$,  
$\on{Sc}(L_{\widehat{\f{g}}}(\ell, 0))$,  
$\on{ScAlg}(V_{\widehat{\f{g}}}(\ell, 0))$, 
and $\on{ScAlg}(L_{\widehat{\f{g}}}(\ell, 0))$  
such that  the following diagram commutes with all 
maps being $G$-equivariant
\[ \xymatrix{\on{ScAlg}(V_{\widehat{\f{g}}}(\ell, 0))\ar[r]\ar[d]&\on{ScAlg}(L_{\widehat{\f{g}}}(\ell, 0))\ar[d] \\
 \on{Sc}(V_{\widehat{\f{g}}}(\ell, 0)) \ar[r] &\on{Sc}(L_{\widehat{\f{g}}}(\ell, 0))}.
\]

We are interested in computing  $G$-orbit structures of the above varieties. 
%The results in the case of $ \f{g}$ suggests that all each of the above  varieties have finitely many $ G$-orbits.
And we will first focus on $ V=V_{\widehat{\f{g}}}(\ell, 0)$.  Note that $ V_1=\f{g}$ as a Lie algebra and $ V_n$ is a module over the Lie algebra $V_1$ under the action $a\cdot v=a_{0}(v)$ for all $v\in V_n$.
Recall that $ a(n)=a\otimes t^n \in \widehat{\f{g}}$.
We will repeatedly use the following formulas \cite[(6.2.67)]{LL}
\begin{equation} a(n)b(m)=b(m)a(n)+[a,b](n+m) +n\delta_{n+m, 0}\< a,b\> K; \end{equation} 
\begin{equation}\label{Ln:comm}
 [L(n), \beta(m)]=-m\beta(m+n)  ,\;  \text{ for all }\beta\in \f{g}, \; \text{ for all } m, n \in \bb{Z}.
\end{equation}

Since  $[L(n), a(0)]=0$ for all $ a\in \f{g}$, each $ L(n): V_m\rightarrow V_{m-n}$ is a $ \f{g}$-module homomorphism.
And since $ L(1)a(-2)\mbf{1}=2a(-1)\mbf{1}$. Thus $ L(1): V_2\rightarrow V_1$ is a surjective $\f{g}$-module homomorphism.
Thus in $V_2$, the vector subspace spanned by $ \{ u^i(-2)\mbf{1} \;|\; 1\leq i \leq \dim \f{g} \}$ as a $\f{g}$-module is isomorphic to $\f{g}$ and the quotient of $ V_{2}$ by this submodule is a $\f{g}$-module isomorphic to the second symmetric power $ \on{Sym}^2(\f{g})$. Since $ V_2$ is finite dimensional, thus
we have $ V_2$ isomorphic to $ \f{g}\oplus  \on{Sym}^2(\f{g})$ as a 
$\f{g}$-module  and $\on{Ker}L(1)\cap V_2$ is a $ \f{g}$-submodule. 
Thus the set of all semi-conformal vectors are solutions of the equaltion 
$L'(0)\omega'=2\omega'$.  
%Next we describe $ \ker(L(1))\cap V_2$.
%\begin{prop}  $\ker(L(1))\cap V_2$ is a $ \f{g}$-submodule and isomorphic to the $\on{Sym}^2(\f{g})$.
%\end{prop}

Choose an orthonormal basis $ \{u^1, \cdots, u^d\}$ of $ \f{g}$ with respect   to  the Killing form and we define the structure constants $R=(\gamma_{ij}^{s})$
 with respect to this basis by
\[ [u^i, u^j]=\sum_{s=1}^d \gamma_{ij}^{s}u^s.\]
Then for each fixed $s$, $R^s$ is a skew symmetric matrix.
We will use the following conventions
\begin{itemize}
\item[(1)]  For each pair $(j, s) $,  $R^s_{*j}=[\gamma^s_{1j}, \gamma^s_{2j},\cdots, \gamma^s_{dj}]$ is a {\em row vector};

\item[(2)] For each pair $(j, k) $,  $R_{k*}^{j}=[\gamma_{k1}^j,\gamma_{k2}^j, \cdots, \gamma_{kd}^j]^{tr}$ is a {\em column vector};

\item[(3)] For each $m$,  $R^m=(\gamma_{ij}^m)$ is a {\em matrix} and
 $R_{m*}^{*}=(\gamma_{mi}^{j})$ is a {\em matrix} with the $(i,j)$ entry being $  \gamma_{mi}^{j}$;

\item[(4)] For a matrix $M$, we denoted the $j-$th column by $M_j$ and the $j-$th row by $M^j$ and $M_{ij}$ for the $(i,j)$-entry of $M$;

\item[(5)] Let $A$ be the  matrix with the $(i,j)$ entry being $a_{ij}$. For each $i$, $AR^i$ is a matrix with the $(r,s)$ entry being $(AR^i)_{rs}=\sum\limits_{h=1}^da_{rh}\gamma_{hs}^i$;

\item[(6)] Let $A$ be the matrix with the $(i,j)$ entry being $a_{ij}$. We denote a matrix with the $(i,j)$ entry being $\sum\limits_{r,s=1}^d(AR^i)_{rs}(AR^j)_{sr}$ by  $\left(\sum\limits_{r,s=1}^d(AR^i)_{rs}(AR^j)_{sr}\right)$;

\item[(7)] Let $A$ be the matrix with the $(i,j)$ entry being $a_{ij}$. We denote the matrix with the $(i,j)$ entry being $\sum\limits_{r,s=1}^d(AR^s)_{ir}(AR^j)_{rs}$ by $\left(\sum\limits_{r,s=1}^d(AR^s)_{ir}(AR^j)_{rs}\right)$.

\end{itemize}

%\section{The semi-conformal vectors of   $V_{\widehat{\sl_2}}(k,0)$ for $k\neq -2$}
%\def\theequation{3.\arabic{equation}}
%\setcounter{equation}{0}

%For the vertex operator algebra $V_{\widehat{\sl_2}}(k,0)$ for $k\neq -2$,  we know that $V_{\widehat{\sl_2}}(k,0)_1$ is a Lie algebra with the Lie bracket $[u,v]=u_0v$ for $u,v\in V_{\widehat{\sl_2}}(k,0)_1$, and  $<u,v>\mathbf{1}=u_1v$ for $u,v\in V_{\widehat{\sl_2}}(k,0)_1$ is a bilinear form on $V_{\widehat{\sl_2}}(k,0)_1$.  As a Lie algebra, we  know that
%$V_{\widehat{\sl_2}}(k,0)_1\cong \sl_2(\C)$.
%Using the same symbols in Section 2, we know that $\alpha_1(-1)\mathbf{1},\alpha_2(-1)\mathbf{1}, \alpha_3(-1)\mathbf{1}$ are an orthonormal basis of  $V_{\widehat{\sl_2}}(k,0)_1$.  Let $\{\gamma_{i,j}^l|i,j,l=1,2,3\}$ be the structural constants of  the basis $\{\alpha_1(-1)\mathbf{1},\alpha_2(-1)\mathbf{1}, \alpha_3(-1)\mathbf{1}\}$. According to (\ref{e2.3}),  we have only nonzero structural constants $\gamma_{12}^3=\gamma_{23}^1=\sqrt{-2},\gamma_{13}^2=-\sqrt{-2}$.

%For the vertex operator algebra $V_{\widehat{\sl_2}}(k,0)$, the homogeneous subspace
% $V_{\widehat{\sl_2}}(k,0)_2$  has
%a basis
%\begin{equation}\label{e4.1} \{\alpha_i(-1)\alpha_j(-1)\mathbf{1}, \alpha_s(-2)\mathbf{1}|1\leq i\neq j\leq 3, s=1,2,3\}.\end{equation}
We use  $\mbf{u}(-1)=[u^1(-1), \cdots, u^d(-1)]$ to denote an ordered basis of $ \f{g}\otimes t^{-1}\subseteq \widehat{\f{g}}$.  In the following proposition, we give an explicit description of the  $ \f{g}$-submodule $  \on{Ker}L(1)\bigcap  V_{\widehat{\f{g}}}(\ell,0)_2$.
\begin{prop}\label{p3.1} Elements
$\omega'\in \on{Ker}L(1)\bigcap  V_{\widehat{\f{g}}}(\ell,0)_2 $  correspondes to   a unique symmetric $ d\times d$-matrix $A$ such that $\omega'=\mbf{u}(-1)A\mbf{u}(-1)^{tr}\mathbf{1}$. Or equivalently,  $\on{Ker}L(1)\bigcap  V_{\widehat{\f{g}}}(\ell,0)_2 $ isomomorphic to the space of all quadratic forms on $\f{g}^*$ as $G$-modules.
\end{prop}
\begin{proof}
We can write $\omega'\in V_{\widehat{\f{g}}}(\ell,0)_2$ as
\begin{equation} \omega'=\sum_{i,j=1}^{d}a_{ij}u^i(-1)u^{j}(-1)\mathbf{1}+\sum_{s=1}^{d}b_su^s(-2)\mathbf{1}.\end{equation}
Using $L(1)\mathbf{1}=u^{i}(0)\mathbf{1}=0$, we have 
$$\begin{array}{llllllllllll}
L(1)\omega'&=\sum\limits_{i,j=1}^{n}a_{ij}L(1)u^i(-1)u^j(-1)\mathbf{1}
+\sum\limits_{s=1}^{n}b_sL(1)u^s(-2)\mathbf{1}\\
&=\sum\limits_{i,j=1}^{n}a_{ij}[L(1),u^i(-1)]u^j(-1)\mathbf{1}+\sum\limits_{i,j=1}^{n}a_{ij}u^i(-1)[L(1),u^j(-1)]\mathbf{1}
+\sum\limits_{s=1}^{n}b_s[L(1),u^s(-2)]\mathbf{1}\\
&=\sum\limits_{i,j=1}^{n}a_{ij}u^i(0)u^j(-1)\mathbf{1}+\sum\limits_{i,j=1}^{n}a_{ij}u^i(-1)u^j(0)\mathbf{1}+2\sum\limits_{s=1}^{n}b_su^s(-1)\mathbf{1}\\
&=\sum\limits_{i,j=1}^{n}a_{ij}[u^i(0),u^j(-1)]\mathbf{1}+\sum\limits_{i,j=1}^{n}a_{ij}u^j(-1)u^i(0)\mathbf{1}+2\sum\limits_{s=1}^{n}b_su^s(-1)\mathbf{1}\\
&=\sum\limits_{i,j=1}^{n}a_{ij}[u^i(0),u^j(-1)]\mathbf{1}+2\sum\limits_{s=1}^{n}b_su^s(-1)\mathbf{1}\\
&=\sum\limits_{i,j,l}^{n}a_{ij}\gamma_{ij}^lu^{l}(-1)\mathbf{1}
+2\sum\limits_{s=1}^{n}b_su^s(-1)\mathbf{1}\\
&=\sum\limits_{l=1}^{n}(\sum\limits_{i,j}^{n}a_{ij}\gamma_{ij}^l+2b_l)u^{l}(-1)\mathbf{1}\\
&=0.
\end{array}
$$
This implies $\sum\limits_{i,j=1}^{d}a_{ij}\gamma_{ij}^l+2b_l=0.$ Noting that $ \gamma_{ii}^l=0$ and $\gamma_{ij}^{l}=-\gamma_{ji}^{l}$, we have
$$
\begin{array}{llllll}
b_l&=-\frac{1}{2}\sum\limits_{i,j=1}^{d}a_{ij}\gamma_{ij}^l
=-\frac{1}{2}(\sum\limits_{1\leq i<j\leq d}a_{ij}\gamma_{ij}^l+
\sum\limits_{d\geqq i>j\geqq 1}^{n}a_{ij}\gamma_{ij}^l)\\
&=-\frac{1}{2}\sum\limits_{1\leq i<j\leq d}(a_{ij}-a_{ji})\gamma_{ij}^l.
\end{array}
$$
If the matrix $ A=(a_{ij})$ is symmetric, then $ \mbf{u}(-1)A\mbf{u}(-1)^{tr}\mbf{1} \in \ker(L(1))$.

The assignment $ A \mapsto \mbf{u}(-1)A\mbf{u}(-1)^{tr}\mbf{1} \in \ker(L(1))$ defines a linear map from the space of  quadratic forms $\on{Sym}^2(\f{g})$ to $\ker(L(1))\cap V_2$, which is injective. Comparing their dimensions give an isomorphism. Note that $L(1):V_2\rightarrow V_1\cong \f{g}(-1)\mathbf{1}$ is surjective.

 Note that $\on{Sym}^2(\f{g})=\bb{C}[\f{g}^*]_{2}$  is the space of all degree two symmetric polynomial functions on the dual vector space $ \f{g}^*$.
\end{proof}
 From now on we can write elements in $ \ker(L(1))\cap V_2$ in terms of
 $\omega'=\mbf{u}(-1)A\mbf{u}(-1)^{tr}\mbf{1}$ with $A$ being a symmetric matrix
 and write
 $Y(\omega', z)=\sum\limits_{n\in \Z}L'(n)z^{-n-2}$.  Similar to the computation  of  Proposition 3.1, we can obtain the following equations in Proposition 3.2--3.4.
 \begin{prop}\label{p3.2} For  $\omega'=\mbf{u}(-1)A\mbf{u}(-1)^{tr}\mathbf{1}$, $L'(1)\omega'=0$ if and only if
\begin{equation}\label{e3.4}
\sum_{s=1}^{d}\sum_{j=1}^{d}\left(R_{*j}^s\left[\sum_{l=1}^{d}(AR^lA)R_{l*}^j+\ell(A^2)_{j}\right]\right)u^s(-1)\mbf{1}=0.
\end{equation}
 \end{prop}
 %\begin{proof}
%By the formula (3.3), we know that
%$$L'(1)=\sum\limits_{i,j=1}^{n}a_{ij}\sum\limits_{k\in \Z}:u^i(k)u^j(1-k):-2\sum\limits_{s=1}^nb_su^s(1).$$
%Similar to the computation  of  Proposition 3.1, we can obtain equations (3.4).\end{proof}
\begin{prop}\label{p3.3}
The  vector $\omega'=\mbf{u}(-1)A\mbf{u}(-1)^{tr}\mathbf{1}$ satisfies $L'(0)\omega'=2\omega'$ if and only if
\begin{equation}\label{e3.5}
2\ell A^2+\sum\limits_{m=1}^{d}AR^{m}AR_{m*}^{*}+\left(\sum_{r,s=1}^d(AR^i)_{rs}(AR^j)_{sr})\right)
+\left(\sum_{r,s=1}^{d}(AR^s)_{ir}(AR^j)_{sr})\right)=A,
\end{equation} and 
\begin{equation}\label{e3.6}
\ell A(\sum\limits_{i=1}^{d}R^iA_{i})=0.
\end{equation}
\end{prop}

\begin{prop}\label{p3.4}
The  vector $\omega'=\mbf{u}(-1)A\mbf{u}(-1)^{tr}\mathbf{1}$ satisfies $L(2)\omega'=L'(2)\omega'=\frac{c'}{2}\mathbf{1}$ if and only if the symmetric $A=(a_{ij})$
satisfies the following equations
\begin{equation}\label{e3.7}
c'=2\ell tr(A);
\end{equation}
\begin{equation}\label{e3.8}\ell tr(A)=2\ell^2  tr(A^2) +2\ell \sum\limits_{i,j=1}^{d}((AR)^j(AR)^i)_{(i,j)}-\ell \sum\limits_{i=1}^{d}tr(A^2(R^i)^2).
\end{equation}
 \end{prop}
 Therefore we have
\begin{theorem}\label{t3.5} A vector $\omega'=\mbf{u}(-1)A\mbf{u}(-1)^{tr}\mathbf{1}\in \on{Sc}(V_{\widehat{\f{g}}}(\ell,0),\omega)$ if and only if the symmetric matrix $A=(a_{ij})$ satisfies equations
(\ref{e3.4})--(\ref{e3.8}).\end{theorem}

\begin{remark}\label{r3.6} If $\omega'\neq 0$ is in $ \on{Sc}(V_{\widehat{\f{g}}}(\ell,0),\omega)$, then for $ \alpha\in \bb{C}$, $ \alpha\omega' \in  \on{Sc}(V_{\widehat{\f{g}}}(\ell,0),\omega)$ if and only if $\alpha =0, 1$. This follows from   \eqref{e3.5}.
\begin{remark} Here, we can  generalize to more general  cases of vertex operator algebra     $(V,\omega)$.  Let $(V,\omega)$ be a CFT-type vertex operator algebra (\cite{DLMM,DM}) with the assumption that  $V_2$ (the  subspace  with degree 2 of $V$) is generated by $V_1$ (the subspace with degree 1 of $V$). It is well-known that  $V_1$ is a Lie algebra with the Lie bracket $[u,v]=u_0v$ for $u,v\in V_1$ and  a bilinear form $\<u,v\>\mathbf{1}=u_1v$ for $u,v\in V_1$.  At first, we  choose a basis of the Lie algebra $V_1$.  Then a vector belongs to $\on{Ker}L(1)\cap V_2$ if and only if there exists a symmetric matrix such that  it can be also  expressed  as a quadratic form  with respect to  the fixed  basis of $V_1$.  This proposition  is a key for us to compute the sufficient and  necessary condition  semi-conformal vectors of $V$ satisfy.  With the help of this proposition,  we also obtain the  matrix equations of   semi-conformal vectors  of $V$ satisfy as similar as  the equations in   Theorem \ref{t3.5}. \end{remark}
The automorphism group $\on{Aut}(\f{g})$ acts on the set $ \on{Sc}(V_{\widehat{\f{g}}}(\ell,0),\omega)$ is given by matrix congruence, i.e., if $ \sigma\in \on{Aut}(\f{g})$ has matrix $g=(g_{ij})$. Then $\sigma(\omega)$ corresponds to the symmetric matrix $gAg^{tr}$. This follows from
\[\sigma(\mbf{u}(-1)A\mbf{u}(-1)^{tr}\mathbf{1})=\sigma(\mbf{u}(-1))A\sigma(\mbf{u}(-1)^{tr})\mathbf{1}=(\mbf{u}(-1)g)A(g^{tr}\mbf{u}(-1)^{tr})\mathbf{1}.\]

Note that $ gR^sg^{tr}=R^s$ for all $s$, and the matrix $g\in \on{SO}_{d}(\f{g})$. Thus each of the equations \eqref{e3.4}--\eqref{e3.8} are invariant under $g$. Therefore the question becomes a question purely on symmetric matrices. Under the congruence action by the group $\on{Aut}(\f{g})$.
Thus the question is to classify the $\on{Aut}(\f{g})$-orbits in the set of symmetric matrices satisfy the equations \eqref{e3.4}--\eqref{e3.8}.
\end{remark}

According to Proposition \ref{p3.1}, for each $\omega'\in \on{Sc}(V_{\widehat{\sl}_2}(\ell,0),\omega)$, if $A$ is the corresponding symmetric matrix, then we can denote by $\omega_A=\omega'$. From the above statements,  we have
\begin{prop} \label{p3.8}
 The action of $\on{Aut}(\f{g})$ on $\on{Sc}(V_{\widehat{\f{g}}}(\ell,0),\omega)$ is defined as follows:
Fixing  an orthonormal basis $\{u^1,\cdots, u^d\}$ of $\f{g}$,  for each $\omega_A\in \on{Sc}(V_{\widehat{\f{g}}}(\ell,0),\omega)$ and $\forall \sigma\in \on{Aut}(\f{g})$, we have
 $\sigma(\omega_A)=\omega_B\in \on{Sc}(V_{\widehat{\f{g}}}(\ell,0),\omega)$,
 where  $B=gAg^{tr}$ and $g$ is the matrix of $\sigma$ with respect to the basis $\{u^1,\cdots, u^d\}$.
\end{prop}
\begin{example}
Here we give some examples of semi-conformal vectors of $(V_{\widehat{\sl}_2}(\ell,0),\omega)$. Of course, $\omega$ is a semi-conformal vector with
the corresponding symmetric matrix $$
A=
\begin{pmatrix}
\frac{1}{2(\ell+2)}& 0& 0\\
 0& \frac{1}{2(\ell+2)}& 0\\
 0& 0& \frac{1}{2(\ell+2)}
\end{pmatrix}.
$$

Let $\h=\C u^3$. It generates a Heisenberg vertex operator algebra $V_{\widehat{\eta}}(\ell,0)$ in $(V_{\widehat{\sl}_2}(\ell,0),\omega)$ with
the conformal vector $\omega_{\h}=\frac{1}{2\ell}u^3(-1)^2
\mathbf{1}\in \on{Sc}(V_{\widehat{\sl}_2}(\ell,0),\omega)$.
Its corresponding symmetric matrix is
$$
A_{\h}=\left(
\begin{array}{lllll}
0\ \ 0\ \ 0\\
0\ \ 0\ \ 0\\
0\ \ 0\ \frac{1}{2\ell}
\end{array}
\right).
$$
The commutant of  $V_{\widehat{\h}}(\ell,0)$ in $(V_{\widehat{\sl}_2}(\ell,0),\omega)$ is
well-known as a parafemion vertex operator algebra(cf. \cite{DLY,DW1,DW2}), and its corresponding symmetric matrix is $A-A_{\h}.$
\end{example}
Denote by
 $$
 \begin{array}{llll}
 \on{MinSc}(V_{\widehat{\f{g}}}(\ell,0),\omega)=\left\{
 0\neq \omega'\in \on{Sc}(V_{\widehat{\f{g}}}(\ell,0),\omega)|
\begin{array}{llll} \forall 0\neq \omega''\in \on{Sc}(V_{\widehat{\f{g}}}(\ell,0),\omega)~such~that~\\
\omega''\preceq \omega', ~then ~\omega'=\omega''
\end{array}\right\}.
 \end{array}$$
 \begin{prop}\label{p3.9}
 Let $(V,\omega)$ be a vertex operator algebra with conformal vector $\omega$. For for any
 $\omega',\omega''\in \on{Sc}(V,\omega)$ we have $\w'\preceq \w''$ if and only if  $\sigma(\omega')\preceq \sigma(\omega'')$  for each $\sigma\in \on{Aut}(V,\w)$. In particular, if $\w'\in \on{Min}\on{Sc}(V,\omega)$, then $\sigma(\omega')\in \on{Min}\on{Sc}(V,\omega)$ for each $\sigma\in   \on{Aut}(V,\w)$. 
  \end{prop}
\begin{proof} For any $ \sigma \in \on{Aut}(V)$, we have 
\[ [Y(u, z_1), Y(v, z_2)]=0 \; \text{ if and only if }[Y(\sigma(u), z_1), Y(\sigma(v), z_2)]=0\]
for all $ u, v \in V$. Therefore, 
\[ \sigma(C_V(U))=C_V(\sigma(U))\]
for any vertex subalgebra $U$. Hence the proposition follows by considering the vertex operator subalgebras $ \langle \omega'\rangle $ and $ \langle \omega''\rangle $.
\end{proof}
 
%If $\omega'\preceq \omega''$ in $\on{Sc}(V,\omega)$, i.e,
% there exists a semi-conformal vertex operator  subalgebra
%$(U, \omega'')$ of $V$ such that $\omega'\in \on{Sc}(U, \omega'')$. 
%For each $ \sigma\in\on{Aut}(V,\omega)$, 
%we have
%$$\sigma(\omega')\in \sigma
%(U,\w'')\subset C_{V}(C_{V}
%(\<\omega''\>).$$
%Since $\omega'\preceq\omega''$, we write $Y(\w',z)=\sum_{n\in \Z}L'(n)z^{-n-2}$ and   $Y(\w'',z)=\sum_{n\in \Z}L''(n)z^{-n-2}$, then
%$L'(n)=L''(n)$ on $C_{U}(C_{U}(\<\omega'\>))$ for $n\geq -1$
%(cf.  \cite[Proposition 2.7]{CL}).
%So $\sigma(L'(n))=\sigma(L''(n))$ on
%$\sigma(C_{U}(C_{U}(\<\omega'\>)))$ for $n\geq -1$,
%hence  $\sigma(\omega')$ is a semi-conformal vector of
%$(\sigma(U),\sigma(\w''))$, i.e.,
%$\sigma(\omega')\preceq \sigma(\omega'')$.
%
%Let $\omega'\in \on{MinSc}(V,\omega)$. If~
%$\sigma(\omega')\notin  \on{MinSc}(V,\omega)$,
%there exists $\omega''\in \on{MinSc}(V,\omega)$
%such that  $\omega''\prec \sigma(\omega')$.  And then  there exists a unique  preimage
%$\sigma^{-1}(\omega'')\in \on{Sc}(V,\omega)$ such that
%$\sigma^{-1}(\omega'')\prec \omega'$. This is a contradiction, since $\omega'\in \on{MinSc}(V,\omega)$.
%\end{proof}
\begin{prop} \label{p3.10}
When  $\ell$ is a positive integer, we have
$$\on{Sc}(V_{\widehat{\f{g}}}(\ell,0),\omega)=\on{Sc}(L_{\widehat{\f{g}}}(\ell,0),\omega).$$
\end{prop}
\begin{proof}
When $\ell$ is a positive integer,  $L_{\widehat{\f{g}}}(\ell,0)=V_{\widehat{\f{g}}}(\ell,0)/<e_{\theta}^{\ell+1}(-1)\mathbf{1}>$, where $\theta$ is the highest root of $\f{g}$. If $\ell\geq 2$, then  we note that $\on{Sc}(V_{\widehat{\f{g}}}(\ell,0),\omega)\cap <e_{\theta}^{\ell+1}(-1)\mathbf{1}>=\{0\}$. By Lemma \ref{l2.8}.\ 2), we have $\on{Sc}(V_{\widehat{\f{g}}}(\ell,0),\omega)\hookrightarrow  \on{Sc}(L_{\widehat{\f{g}}}(\ell,0),\omega)$. Actually, since $V_{\widehat{\f{g}}}(\ell,0)_2\cap <e_{\theta}^{\ell+1}(-1)\mathbf{1}>=\{0\}$, then  $V_{\widehat{\f{g}}}(\ell,0)_2=L_{\widehat{\f{g}}}(\ell,0)_2$, so we have
$\on{Sc}(V_{\widehat{\f{g}}}(\ell,0),\omega)=\on{Sc}(L_{\widehat{\f{g}}}(\ell,0),\omega).$
If $\ell=1$, $V_{\widehat{\f{g}}}(\ell,0)_2\cap <e_{\theta}^{2}(-1)\mathbf{1}>=\C e_{\theta}^{2}(-1)\mathbf{1}$. Obviously,  $e_{\theta}^{2}(-1)\mathbf{1}\notin \on{Sc}(V_{\widehat{\f{g}}}(\ell,0),\omega)$, then $\on{Sc}(V_{\widehat{\f{g}}}(\ell,0),\omega)\cap <e_{\theta}^{\ell+1}(-1)\mathbf{1}>=\{0\}$. By Lemma \ref{l2.8}.\ 2), we have $\on{Sc}(V_{\widehat{\f{g}}}(\ell,0),\omega)\hookrightarrow  \on{Sc}(L_{\widehat{\f{g}}}(\ell,0),\omega)$.  Since $e_{\theta}^{2}(-1)\mathbf{1}\notin \on{Sc}(V_{\widehat{\f{g}}}(\ell,0),\omega)$, then $\on{Sc}(V_{\widehat{\f{g}}}(\ell,0),\omega)=\on{Sc}(L_{\widehat{\f{g}}}(\ell,0),\omega)$.
\end{proof}

\section{Case $A^{(1)}_1$}
\subsection{}
Set $ \f{g}=\f{sl}_2$. Then the corresponding group $G=\on{PSL}_2(\bb{C})=\on{SL}_2(\bb{C})/\{\pm I\}$ acts on $\f{g}$ by conjugation of matrices. Let ${h, e, f}$ be a standard Chevalley basis of $\sl_2$. The normalized Killing form $\langle \cdot, \cdot \rangle $ on $\f{g}$ with respect to this Chevalley basis has the form
\[ \langle h,h \rangle =2, \langle e,f\rangle =1, \langle h,e\rangle =\langle h,f\rangle =\langle e,e\rangle=\langle f,f\rangle =0.\]
The dual Coexter number $h^V=2$ (the notation $h^V$ will not cause confusion in this section since we will simply use $2$ for the dual Coexter number leaving $h$ for the standard basis for $\f{sl}_2$). Thus the conformal vector $ \omega $ of $ V_{\widehat{\f{g}}}(\ell, 0)$ is
\begin{eqnarray}
\omega&=\frac{1}{2(\ell+2)}(\frac{1}{2}h(-1)^2\mathbf{1}+e(-1)f(-1)\mathbf{1}+f(-1)e(-1)\mathbf{1})\notag \\
&=\frac{1}{2(\ell+2)}(\frac{1}{2}h(-1)^2\mathbf{1}+2e(-1)f(-1)\mathbf{1}-h(-2)\mathbf{1}).
\end{eqnarray}
The central charge of $ V=V_{\widehat{\f{g}}}(\ell, 0)$ is ${3\ell}/{(\ell+2)}$.
Let $\cal{J}\subseteq  V_{\widehat{\f{g}}}(\ell, 0)$ be the unique maximal $\widehat{\f{g}}$-submodule such that $ L_{\widehat{\f{g}}}(\ell, 0)=V_{\widehat{\f{g}}}(\ell, 0)/\cal{J}$. Note that  when $\ell \in \bb{Z}_{\geq 0}$, $\cal{J}$ is the $\widehat{\f{g}}$-submodule generated by $ e(-1)^{\ell+1}\mbf{1}$. In particular, when $ \ell>1$, then $\cal{J}\subseteq \oplus_{n>2}V_{n}$.

We can give an orthonormal basis
with respect to the normalized Killing form on $\sl_2$:
$$\alpha_1=\frac{e+f}{\sqrt{2}}=
\frac{1}{\sqrt{2}}\begin{pmatrix} 0& 1\\1&0\end{pmatrix}
;\;
\alpha_2=\frac{e-f}{\sqrt{-2}}=\frac{1}{\sqrt{2}}\begin{pmatrix} 0& -\mbf{i}\\\mbf{i}&0\end{pmatrix}; \; \alpha_3=\frac{h}{\sqrt{2}}=\frac{1}{\sqrt{2}}\begin{pmatrix} 1& 0\\0&-1\end{pmatrix}.$$
The structure constants $R=(\gamma_{ij}^s)$ is given in the following
\begin{equation}\label{e2.3}[\alpha_1,\alpha_2]=\sqrt{2}\mbf{i}\alpha_3; \;[\alpha_2,\alpha_3]
=\sqrt{2}\mbf{i}\alpha_1;\;  [\alpha_3,\alpha_1]=\sqrt{2}\mbf{i}\alpha_2.\end{equation}
Thus $ \gamma_{ij}^s=\sqrt{2}\mbf{i}$ if $ (ijs)$ is one of the three cycles $(123), (231), (312)$ and zero otherwise.
Then the conformal vector of $L_{\widehat{\sl_2}}(\ell, 0)$ can be rewritten as
\begin{equation} \omega=\frac{1}{2(\ell+2)}\sum\limits_{i=1}^{3}\alpha_i(-1)^2\mathbf{1}.\end{equation}

Since $G=\on{PSL}_2(\bb{C})$ action on $ \f{g}$ preserves the killing form, then $ \on{PSL}_2(\bb{C})\subset \on{SO}_{3}(\bb{C})$.  In fact, since $\sigma\mapsto g$ gives a covering from $\on{SL}_2(\C)$  to $\on{SO}_3(\C)$, then $ \on{PSL}_2(\bb{C})= \on{SO}_{3}(\bb{C})$.
\subsection{}
Let $$
A=\left(
\begin{array}{llll}
&a_1&0&0\\
&0&a_2&0\\
&0&0&a_3
\end{array}
\right)$$ be a diagonal matrix. According to equations \eqref{e3.4}--\eqref{e3.8}, we have  $\omega_A\in \on{Sc}(V_{\widehat{\sl}_2}(\ell,0),\omega)$
if and only if $A$ satisfies
\begin{eqnarray}\label{e4.4}
\left\{
\begin{array}{llllll}
4(a_1a_2+a_1a_3-a_2a_3)+2\ell a_1^2=a_1;\\
4(a_1a_2+a_2a_3-a_1a_3)+2\ell a_2^2=a_2;\\
4(a_1a_3+a_2a_3-a_1a_2)+2\ell a_3^2=a_3.
\end{array}
\right.
\end{eqnarray}

\begin{theorem} \label{t4.1}When $\ell\neq -2,0,1$, we have
$$\on{Sc}(V_{\widehat{\sl}_2}(\ell,0),\omega)=\{
\omega_B| B=oAo^{tr} ~for ~a~ solution ~A ~of~  (\ref{e4.4}) ~and ~o\in \on{SO}_3(\C)\}.
$$\end{theorem}
\begin{proof}
For any $\omega_B\in \on{Sc}(V_{\widehat{\sl_2}}(\ell,0),\omega)$, since $B$ is a symmetric matrix, then there exists an $o\in \on{SO}_3(\C)$ such that $oBo^{tr}=A$ is a diagonal matrix. According to Proposition \ref{p3.8},  there exists a $\sigma\in \on{PSL}_2(\C)$
such that $\omega_A=\sigma(\omega_B)$.
Thus, $\omega_{A}$ is a semi-conformal vector of $\on{Sc}(V_{\widehat{\sl}_2}(\ell,0),\omega)$.
Hence $A$ is a solution of equations (\ref{e4.4}).

If $B=oAo^{tr}$ for a solution $A$ of equations (\ref{e4.4}) and $o\in \on{SO}_3(\C)$,  according to Proposition \ref{p3.8}, there exists a $\sigma\in \on{PSL}_2(\C)$ such that $\sigma(\omega_A)=\omega_B$. So
$\omega_B\in \on{Sc}(V_{\widehat{\sl}_2}(\ell,0),\omega)$.
\end{proof}
When $\ell=0$, we solve equations (\ref{e4.4}) and they only have one resolution $a_1=a_2=a_3=\frac{1}{4}$, i.e., there is only a diagonal matrix
$\frac{1}{4}I$, which is just the corresponding semi-conformal vector $\omega$,
so $\on{Sc}(V_{\widehat{\sl}_2}(\ell,0),\omega)=\{0,\omega\}$; From \cite[6.6]{LL}, we have
$L_{\widehat{\sl}_2}(0,0)=\C$.
Hence $\on{Sc}(L_{\widehat{\sl}_2}(0,0),0)=\{0\}$.

When $\ell=4$, we shall discuss solutions of equations (\ref{e4.4}) and study $\on{PSL}_2(\C)$-orbits of  $\on{Sc}(V_{\widehat{\sl}_2}(4,0),\omega)$ and  $\on{Sc}(L_{\widehat{\sl}_2}(4,0),\omega)$ in next section.

When $\ell\neq -2, 0, 4$, the equations (\ref{e4.4})  have exactly six matrix solutions $A_1,\cdots,A_6$ as follows
$$
A_1=\left(
\begin{array}{llll}
&\frac{1}{2\ell}&0&0\\
&0&0&0\\
&0&0&0
\end{array}
\right),
A_2=s_1A_1s_1^{tr},
A_3=s_2A_1s_2^{tr};
$$
$$
A_4=\left(
\begin{array}{llll}
&-\frac{1}{\ell(\ell+2)}&\ \ \ \ 0&\ \ \ \ 0\\
&\ \ \ \ 0&\frac{1}{2(\ell+2)}&\ \ \ \ 0\\
&\ \ \ \ 0&\ \ \ \ 0&\frac{1}{2(\ell+2)}
\end{array}
\right),
A_5=s_1A_4s_1^{tr},
A_6=s_2A_4s_2^{tr},
$$
where 
$$s_1=\left(
\begin{array}{llll}
\ \ 0&1&0\\
-1&0&0\\
\ \ 0&0&1
\end{array}
\right),s_2=\left(
\begin{array}{llll}
\ \ 0&0&1\\
\ \ 0&1&0\\
-1&0&0
\end{array}
\right)\in \on{SO}_3(\C).$$
Let $\{\omega_{A_i}|i=1,\cdots, 6\}$ be semi-conformal vectors determined by diagonal matrices which form a subset of
$\on{Sc}(V_{\widehat{\sl}_2}(\ell,0),\omega)$. Denote by
$$\on{Orb}_{A_i}=\{\omega_B|B=oAo^{tr} ~for~all~o\in \on{SO}_3(\C)\}$$for $i=1,\cdots, 6$,
which are $\on{PSL}_2(\C)$-orbits of $\omega_{A_i}$, respectively.

\begin{cor} \label{c4.2}When the level $\ell\neq -2, 0, 1,4$,
$$\on{Sc}(V_{\widehat{\sl}_2}(\ell,0),\omega)=\on{Orb}_{A_1}\cup \on{Orb}_{A_4}\cup\{0\}\cup\{\omega\}.$$
In particular, when $\ell\in \Z_+$ and $\ell\notin 1,4$,
$$\on{Sc}(L_{\widehat{\sl}_2}(\ell,0),\omega)=\on{Orb}_{A_1}\cup \on{Orb}_{A_4}\cup\{0\}\cup\{\omega\}.$$\end{cor}
\begin{proof}
According to Proposition \ref{p3.8}, there exist a $\sigma_1\in  \on{PSL}_2(\C)$ 
defined $\alpha_1\mapsto -\alpha_2;\alpha_2\mapsto \alpha_1;\alpha_3\mapsto \alpha_3$ such that $\sigma_1(\omega_{A_1})=\omega_{A_2}$ and $\sigma_1(\omega_{A_4})=\omega_{A_5}$. As the same reason, there exist a $\sigma_2\in  \on{PSL}_2(\C)$ 
defined $\alpha_1\mapsto -\alpha_3;\alpha_2\mapsto \alpha_2;\alpha_3\mapsto \alpha_1$ such that $\sigma_1(\omega_{A_1})=\omega_{A_3}$ and $\sigma_1(\omega_{A_4})=\omega_{A_6}$. So $\on{Orb}_{A_1},\on{Orb}_{A_2}$ and $\on{Orb}_{A_3}$ are the same orbit, $\on{Orb}_{A_4}, \on{Orb}_{A_5},\on{Orb}_{A_6}$ are the same orbit. Thus, we can get desired  $\on{Sc}(V_{\widehat{\sl}_2}(\ell,0),\omega)=\on{Orb}_{A_1}\cup \on{Orb}_{A_4}\cup\{0\}\cup\{\omega\}$. By Proposition \ref{p3.10}, when $\ell\in \Z_+$ and $\ell\notin 1,4$, 
we have also $$\on{Sc}(L_{\widehat{\sl}_2}(\ell,0),\omega)=\on{Orb}_{A_1}\cup \on{Orb}_{A_4}\cup\{0\}\cup\{\omega\}.$$
\end{proof}

\begin{prop}\label{p4.3} If $\ell=1$,  we have 
$$\on{Sc}(V_{\widehat{\sl}_2}(\ell,0),\omega)=\on{Sc}(L_{\widehat{\sl}_2}(\ell,0),\omega)=\{\omega,0\}.$$
\end{prop}
\begin{proof}
According to Proposition \ref{p3.1}, for each semi-conformal vector $\omega_A$ of $V_{\widehat{\sl}_2}(1,0)$,
it can be determined by a symmetric
matrix $A$. Since any symmetric matrix can be diagonalized by an orthogonal matrix. Just for
every automorphism $\sigma$ of $V_{\widehat{\sl}_2}(1,0)$, its action on $\omega_A$ will give a semi-conformal
vector $\omega_{oAo^{tr}}$ for  $o\in \on{SO}_3(\C)$. Thus, any semi-conformal vector $\omega_A$
can be changed into a semi-conformal vector determined by a diagonal matrix by using some automorphism of $V_{\widehat{\sl}_2}(1,0)$.
According to \cite[Proposition 13.10]{DL}, two conformal vectors are identified  in vertex operator algebras $(V_{\widehat{\h}}(1,0),\omega_{\h}=\frac{1}{2}\alpha_3(-1)^2\mathbf{1})$  and $(V_{\widehat{\sl}_2}(1,0), \omega)$,  so  $(V_{\widehat{\h}}(1,0),\omega_{\h})\hookrightarrow (V_{\widehat{\sl}_2}(1,0),\omega)$ is a conformal embedding,
that is, $\omega_{A_3}=\omega$, where $V_{\widehat{\h}}(1,0)$ is Heisenberg VOA generated by $\alpha_3$.
Since $\omega_{A_1}=\omega_{M_1A_3M_1^{tr}}$ and $\omega_{A_2}=\omega_{M_2A_3M_2^{tr}}$ for some $M_1,M_2\in \on{SO}_3(\C)$.
Hence $\omega_{A_4}=\omega_{A_5}=\omega_{A_6}=0$.  By  Proposition \ref{p3.10}, $V_{\widehat{\sl}_2}(1,0)$ and $L_{\widehat{\sl}_2}(1,0)$ have both  only two trivial
semi-conformal vectors $\omega$ and $0$.
\end{proof}
\begin{remark}
In general,  one question is to  study the $G$-orbit structure of the variety $\on{Sc}(L_{\widehat{\f{g}}}(1,0),\omega)$ for  the simple affine vertex operator algebra $L_{\widehat{\f{g}}}(1,0)$  associated to the finite dimensional simple Lie algebra $\f{g}$  of $A,D,E$ types.  In these cases,  we note that  $L_{\widehat{\f{g}}}(1,0)\cong V_{Q_{\f{g}}}$ as vertex operator algebras, and then we can study  $G$-orbit informations of $\on{Sc}(L_{\widehat{\f{g}}}(1,0),\omega)$ by  virtue of   the lattice vertex operator algebra $V_{Q_{\f{g}}}$ associated to the root lattice $Q_{\f{g}}$ of $\f{g}$.
\end{remark}
Next, we  consider how to determine the order relation of two semi-conformal
vectors of $(V_{\widehat{\sl}_2}(\ell,0),\omega)
$ by their corresponding symmetric matrices. Since any symmetric matrix can be congruent with a diagonal matrix by a special orthogonal matrix, then we only need to consider the order relation of two diagonal matrices.
\begin{prop}\label{p4.5}
For two diagonal matrices $$
A=\left(
\begin{array}{llll}
&a_1&0&0\\
&0&a_2&0\\
&0&0&a_3
\end{array}
\right),
B=\left(
\begin{array}{llll}
&b_1&0&0\\
&0&b_2&0\\
&0&0&b_3
\end{array}
\right),$$ two vectors
$\omega_A,\omega_B\in \on{Sc}(V_{\widehat{\sl}_2}(\ell,0),\omega)$, then $\omega_A\preceq \omega_B$ if and only if
$A, B$ satisfy the following conditions
\begin{eqnarray}\label{e4.5}
\left\{
\begin{array}{lllll}
2(a_{1}b_{2}+a_{2}b_{1}+a_{1}b_{3}+a_{3}b_{1}-a_{2}b_{3}-a_{3}b_{2})+2\ell a_{1}b_{1}
=a_{1};\\
2(a_{1}b_{2}+a_{2}b_{1}+a_{2}b_{3}+a_{3}b_{2}-a_{1}b_{3}-a_{3}b_{1})+2\ell a_{2}b_{2}=a_2;\\
2(a_{1}b_{3}+a_{3}b_{1}+a_{2}b_{3}+a_{3}b_{2}-a_{1}b_{2}-a_{2}b_{1})+2\ell a_{3}b_{3}
=a_{3};\\
2\ell a_1b_1+4b_1a_2+4b_3a_1-4b_3a_2=2\ell a_1^2+4a_1a_2+4a_1a_3-4a_2a_3;\\
2\ell a_2b_2+4b_2a_1-4b_3a_1+4b_3a_2=2\ell a_2^2+4a_1a_2+4a_2a_3-4a_1a_3;\\
2\ell a_3b_3+4b_3a_1+4b_2a_3-4b_2a_1=2\ell a_3^2+4a_1a_3+4a_2a_3-4a_1a_2;\\
b_1a_2-b_2a_1+b_3a_1-b_1a_3+b_2a_3-b_3a_2=0.
\end{array}
\right.
\end{eqnarray}
\end{prop}
\begin{proof} According to Proposition \cite[Proposition 2.8] {CL},   we can verify   relations (\ref{e4.5}) by straightly computations .
\end{proof}
Next, we expect to describe $\on{MinSc}(V_{\widehat{\sl}_2}(\ell,0),\omega)$
for  level $\ell\neq -2,1, 0,4$.
\begin{theorem}
For the level  $\ell\neq -2,0, 1,4$, we have
$$\on{MinSc}(V_{\widehat{\sl}_2}(\ell,0))=\on{Orb}_{A_1}\cup \on{Orb}_{A_4}.$$
In particular, when the level $\ell\in \Z_+$ and $\ell\notin 1,4$, we have
$$\on{MinSc}(L_{\widehat{\sl}_2}(\ell,0))=\on{Orb}_{A_1}\cup \on{Orb}_{A_4}.$$\end{theorem}
\begin{proof}
We can check  $\omega_{A_1}$ and $\omega_{A_4}$ both don't satisfy the relations (\ref{e4.5}) in Proposition \ref{p4.5},
so there is not the partial order relation $\preceq$ for the pair $\omega_{A_1}, \omega_{A_4}$.
According to Proposition \ref{p3.9}, we know that all semi-conformal vectors in a same orbit have the same partial order. Then
$\on{Orb}_{A_1},\on{Orb}_{A_4}\in \on{MinSc}(V_{\widehat{\sl}_2}(\ell,0))$. Using Corollary \ref{c4.2}, we have
$\on{MinSc}(V_{\widehat{\sl}_2}(\ell,0))=\on{Orb}_{A_1}\cup \on{Orb}_{A_4}.$
\end{proof}
\section{Semi-conformal vectors of $V_{\widehat{\sl}_2}(4,0)$}
In this section,  we consider semi-conformal vectors of $(V_{\widehat{\sl}_2}(4,0),\omega)$.
According to equations (\ref{e4.4}), for a diagonal matrix $$
A=\left(
\begin{array}{llll}
&a_1&0&0\\
&0&a_2&0\\
&0&0&a_3
\end{array}
\right),  $$
Note that $\omega_{A}\in \on{Sc}(V_{\widehat{\sl}_2}(4,0))$ if and only if
$A$ satisfies
\begin{eqnarray}\label{e5.1}
%\left\{
%\begin{array}{llllll}
4(a_1a_2+a_1a_3-a_2a_3)+8a_1^2=a_1;\label{e5.1}\\
4(a_1a_2+a_2a_3-a_1a_3)+8a_2^2=a_2;\label{e5.2}\\
4(a_1a_3+a_2a_3-a_1a_2)+8a_3^2=a_3.\label{e5.3}%\end{array}
%\right.
\end{eqnarray}

When $A$ satisfies $a_1=a_2=a_3$, the equations (\ref{e5.1})-(\ref{e5.3})
have two solutions $a_1=a_2=a_3=0$ and $a_1=a_2=a_3=\frac{1}{12}$, i.e., 
$A=0$ and $A=\frac{1}{12}I$, where $I$ is the $3\times 3$ identity matrix. They are corresponding to semi-conformal vectors $0$ and $\omega$ of $(V_{\widehat{\sl}_2}(4,0))$, respectively.
\begin{prop}\label{p5.1}
When $a_1,a_2, a_3$ don't satisfy $a_1=a_2=a_3$, 
the set of solutions of the equations (\ref{e5.1})-(\ref{e5.3}) is equal to 
\begin{equation}\label{e5.4}
\{(a_1,a_2,a_3)\in \C^3|a_1+a_2+a_3=\frac{1}{8}~and ~a_1a_2+a_1a_3+a_2a_3=0\}.
\end{equation}
\end{prop}
\begin{proof}
We solve  the equations (\ref{e5.1})-(\ref{e5.3}) in following cases:
\begin{itemize}
\item[1)]  If only  two numbers of $a_1, a_2,a_3$ are same, we can suppose that 
$a_1=a_2\neq a_3$,  the equations  (\ref{e5.1}) and (\ref{e5.2}) both become
$ 12a_1^2=a_1$, then $a_1=0$ or $a_1=\frac{1}{12}$. 

When  $a_1=0$, by the equation 
(\ref{e5.3}), we have $a_3=\frac{1}{8}$.  So we get a solution $a_1=a_2=0,a_3=\frac{1}{8}$. Since any permutation of $(a_1,a_2,a_3)$ give a solution
of  the equations  (\ref{e5.1}) and (\ref{e5.3}), then $a_1=a_3=0,a_2=\frac{1}{8}$ and 
$a_2=a_3=0,a_1=\frac{1}{8}$ are both solutions of  the equations  (\ref{e5.1}) and (\ref{e5.3}) in this case.

When $a_1=\frac{1}{12}$, by the equation 
(\ref{e5.3}), we have $8a_3^2-\frac{1}{3}a_3-\frac{1}{36}=0$, so $a_3=\frac{1}{12}$ or
$-\frac{1}{24}$. In this case, $a_3=-\frac{1}{24}$. Hence we get a solution $a_1=a_2=\frac{1}{12}, a_3=-\frac{1}{24}$. Since any permutation of $(a_1,a_2,a_3)$ give a solution
of  the equations  (\ref{e5.1}) and (\ref{e5.3}), then $a_1=a_3=\frac{1}{12}, a_2=-\frac{1}{24}$ and $a_2=a_3=\frac{1}{12}, a_1=-\frac{1}{24}$ are another two solutions of  the equations  (\ref{e5.1}) and (\ref{e5.3}) in this case.
\item[2)] Assume that $a_1, a_2 $ and $ a_3 $ are all distinct.  Note that $a_1\neq a_2$, when  (\ref{e5.1}) minus   (\ref{e5.2}), we have $$8(a_1-a_2)(a_1+a_2+a_3)=a_1-a_2,$$
so we get 
\begin{equation}\label{e5.5}
a_1+a_2+a_3=\frac{1}{8}.
\end{equation}
Adding (\ref{e5.1}) and (\ref{e5.2}) to (\ref{e5.3}), 
 we get $$8(a_1a_2+a_1a_3+a_2a_3)+8(a_1^2+a_2^2+a_3^2)=a_1+a_2+a_3,$$
 i.e.,
$$8(a_1+a_2+a_3)^2-8(a_1a_2+a_1a_3+a_2a_3)=a_1+a_2+a_3.$$
Using the relation (\ref{e5.5}), we get 
\begin{equation}\label{e5.6}
a_1a_2+a_1a_3+a_2a_3=0.
\end{equation}
\end{itemize}

The solutions of the equations (\ref{e5.1})-(\ref{e5.3}) given in case 1)  all satisfy   relations (\ref{e5.5}) and (\ref{e5.6}).

Conversely,  when $a_1,a_2, a_3$ don't satisfy $a_1=a_2=a_3$, using relations (\ref{e5.5}) and (\ref{e5.6}), we can also get the equations (\ref{e5.1})-(\ref{e5.3}).
\end{proof}
Let $S_3$ be the permutation group consisting of permutations of $(a_1,a_2,a_3)$.
We have 

\begin{theorem}\label{t5.3}
$$\on{Sc}(V_{\widehat{\sl}_2}(4,0),\omega)=\{0\}\cup\{\omega\}\cup\bigcup_{A\in T}Orb_{\omega_A},$$ where 
$T=\left\{
A=\left(
\begin{array}{llll}
&a_1&0&0\\
&0&a_2&0\\
&0&0&a_3
\end{array}
\right)\Bigg|
\begin{array}{lll} non-scalar~matrix~A~ satisfies \\
~the~equations~(\ref{e5.5}) ~and ~(\ref{e5.6})
\end{array}\right\}\Bigg/S_3.$
\end{theorem}
\begin{proof}
When $A$ is a scalar matrix as a solution of the equations (\ref{e5.1})-(\ref{e5.3}), we know that $A=0$ or $A=\frac{1}{12}I$, they are corresponding to semi-conformal vectors 
$0$ and $\omega$ of $(V_{\widehat{\sl}_2}(4,0),\omega)$, respectively. Hence they give two orbits $\{0\}$ and $\{\omega\}$ of $\on{Sc}(V_{\widehat{\sl}_2}(4,0),\omega)$.

When  $A$ is not a scalar matrix as a solution of the equations (\ref{e5.1})-(\ref{e5.3}),
by Proposition \ref{p5.1},  if and only if $A$ is a solution of the equations (\ref{e5.5}) and (\ref{e5.6}). For any solution $A$ of  the equations (\ref{e5.1})-(\ref{e5.3}), a matrix $B$ given by the action on $A$ of any permutation of $(a_1,a_2,a_3)$ is still a solution of  the equations (\ref{e5.1})-(\ref{e5.3}).  Moreover, $B$ determines  the same orbit with $A$ in $\on{Sc}(V_{\widehat{\sl}_2}(4,0),\omega)$. Hence, we can classify the set of  solutions of the equations (\ref{e5.5}) and (\ref{e5.6}) by the action  of $S_3$. Denote this classified set 
by $T$. Each element $A$ in $T$ determines a unique  orbit $Orb_{\omega_A}$ of $\on{Sc}(V_{\widehat{\sl}_2}(4,0),\omega)$.

Thus, we get the decomposition of orbits of $\on{Sc}(V_{\widehat{\sl}_2}(4,0),\omega)$.
\end{proof}

As a  result of Theorem \ref{t5.3}, we can get immediately

\begin{cor}\label{c5.4}
$$\on{MinSc}(V_{\widehat{\sl}_2}(4,0),\omega)=\bigcup_{A\in T}Orb_{\omega_A},$$ where $T$ is the same as  in Theorem \ref{t5.3}.
\end{cor}
\begin{proof}
According to Proposition \ref{p3.8} and Proposition \ref{p3.9}, we need only to consider semi-conformal vectors determined by diagonal matrices.  Given $$
A=\left(
\begin{array}{llll}
&a_1&0&0\\
&0&a_2&0\\
&0&0&a_3
\end{array}
\right),
B=\left(
\begin{array}{llll}
&b_1&0&0\\
&0&b_2&0\\
&0&0&b_3
\end{array}
\right)$$ such that $\omega_A,\omega_B\in \on{Sc}(V_{\widehat{\sl}_2}(4,0),\omega)$ and they are not equal to $0, \omega$.
Then $$a_1+a_2+a_3=\frac{1}{8}; a_1a_2+a_1a_3+a_2a_3=0~and ~b_1+b_2+b_3=\frac{1}{8};b_1b_2+b_1b_3+b_2b_3=0.$$
According to Proposition \ref{p4.5}, we can check that $A,B$ don't satisfy the relations (\ref{e4.5}), hence $\omega_A$ and $\omega_B$ don't have partial order relation, so $\on{MinSc}(V_{\widehat{\sl}_2}(4,0),\omega)$ consists of nontrivial semi-conformal vectors of $V_{\widehat{\sl}_2}(4,0)$.
\end{proof}

\begin{cor} The set $T$ in Corollary \ref{c5.4} is an algebraic variety isomorphic to the affine line$ \mathbb{A}^1$.  There are infinitely many $\on{SO}_3(\mathbb{C})$-orbits in $ \on{MinSc}(L_{\widehat{\sl}_2}(4,0),\omega)$.
\end{cor} 
\begin{proof} Note that two diagonal matrices $A$ and $B$ are in the same  $\on{SO}(\mathbb{C})$-orbit if any only they are equal up to permutation of the diagonal entries.  Consider the affine variety, $\mathbb{C}^3$, the set $\mathbb{C}^3/S_3$ is identified by the affine variety with the coordinate algebra being the algebra of symmetric polynomials $ \mathbb{C}[ x_1, x_2, x]^{S_3}=\mathbb{C}[e_1, e_2, e_3]$, which is a polynomial algebra. Here $ e_1, e_2, $ and $e_3$ are the elementary symmetric polynomials. Then the set $\on{MinSc}(L_{\widehat{\sl}_2}(4,0),\omega)/G=T$
which is defined by the equations $ e_\frac{1}{8}$ and $ e_2=0$. In particular, $ T$ is an algebraic variety isomorphic to the affine line $\mathbb{A}^1$. 
\end{proof}

\begin{remark}
Comparing to the cases of $ \ell \neq -2,0,1$,  we know that $\on{MinSc}(V_{\widehat{\sl}_2}(4,0),\omega)$ can be expressed as the union of
infinite $\on{SO}_3(\C)$-orbits, i.e.,
$$\on{MinSc}(V_{\widehat{\sl}_2}(4,0),\omega)=\bigcup_{\omega_A\in T}\on{Orb}_{\omega_A},$$
where
$\on{Orb}_{\omega_A}=\{\omega_{oAo^{tr}}|\; o\in \on{SO}_3(\C)\}$ and 
$T$ is the same as  in Theorem \ref{t5.3}.
\end{remark}
\begin{remark}
By Proposition \ref{p3.10}, we know that $$ \on{Sc}(L_{\widehat{\sl}_2}(4,0),\omega)=\on{Sc}(V_{\widehat{\sl}_2}(4,0),\omega)$$
and
$$\on{MinSc}(L_{\widehat{\sl}_2}(4,0),\omega)
=\on{MinSc}(V_{\widehat{\sl}_2}(4,0),\omega).$$
\end{remark}
The following result is obtained from above theorem immediately.
\begin{cor}
For the level $\ell\neq -2,0,1$, there exists a maximal chain of partial order in $\on{Sc}(V_{\widehat{\sl}_2}(\ell,0),\omega)$ (resp. $
\on{Sc}(L_{\widehat{\sl}_2}(\ell,0),\omega)$ for $\ell\in \Z_+$ and $\ell\neq 1$)
$$0=\omega^0\prec\omega^1\prec \omega^2=\omega$$ with the length $2$.
\end{cor}
\begin{remark}
 Our aim is to understand vertex operator algebra theory by the geometric object $\on{Sc}(V,\omega)$.
  For this purpose, at first, we must understand clearly the structures of  $\on{Sc}(V,\omega)$.  From above results related to the affine vertex operator algebra $(V_{\widehat{\sl}_2}(\ell,0),\omega)$, we can find the variety $\on{Sc}(V, \omega)$ is closely related to the level $\ell$ for the affine vertex operator algebra $(V_{\widehat{\sl}_2}(\ell,0),\omega)$, and this work will lead us to describe $\on{Sc}(V,\omega)$  for general affine vertex operator algebras and some even lattice vertex operator algebras. In future work,  we shall focus on cases for general affine vertex operator algebras and some even lattice vertex operator algebras and discuss structures  and properties of the varieties $\on{Sc}(V,\omega)$  to understand  vertex operator algebra theory.
\end{remark}

\end{document}